\documentclass[11pt]{amsart}

\usepackage{mystyle}

\title{Entropy Bounds for Perfect Matchings in Bipartite Hypergraphs}

\author{Tantan Dai}
\email{tdai44@gatech.edu}

\author{Alexander Divoux}
\email{adivoux@princeton.edu}

\author{Tom Kelly}
\email{tom.kelly@gatech.edu}

\address{Georgia Institute of Technology}

\thanks{Research supported by the National Science Foundation under Grant No. DMS-2247078.}

\begin{document}

\begin{abstract}
    A hypergraph is \textit{bipartite with bipartition $(A, B)$} if every edge has exactly one vertex in $A$, and a matching in such a hypergraph is \textit{$A$-perfect} if it saturates every vertex in $A$. We prove an upper bound on the number of $A$-perfect matchings in uniform hypergraphs with small maximum codegree. Using this result, we prove that there exist order-$n$ Latin squares with at most $(n/e^{2.117})^n$ transversals when $n$ is odd and $n \equiv 0 \pmod 3$. We also show that $k$-uniform $D$-regular hypergraphs on $n$ vertices have at most $((1+o(1))q/e^k)^{Dn/k}$ proper $q$-edge-colorings when $q = (1+o(1))D$ and the maximum codegree is $o(q)$.
\end{abstract}

\maketitle

\section{Introduction}\label{sec:intro}
An \textit{order-$n$ Latin square} is an $n \times n$ array of $n$ symbols such that each row and each column contains each symbol exactly once. Throughout this paper we assume that $n$ is sufficiently large for various asymptotic inequalities to hold. A \textit{transversal} in an order-$n$ Latin square is a collection of $n$ cells which do not share any row, column, or symbol. A \textit{partial transversal} is any collection of $k \le n$ such cells. Transversals are central objects in the study of Latin squares with a long history. For instance, in 1782 Euler considered the question of determining for which $n$ does there exist a Latin square which can be decomposed into $n$ disjoint transversals. Another well-known example is the Latin square corresponding to the Cayley table of the cyclic group $\ZZ_n$. The number of transversals in this Latin square has been studied under many guises, such as additive triples of bijections and toroidal semi-queen configurations. Moreover, each orthomorphism of a finite group corresponds to a transversal in its Cayley table.

Two natural lines of inquiry concern the existence and enumeration of transversals in Latin squares. On the existence side, when $n$ is even, the Cayley table of $\ZZ_n$ provides examples of a Latin square with no transversals. Despite this, the Ryser-Brualdi-Stein conjecture \cite{R67, BR91, S75} states that every order-$n$ Latin square has a partial transversal of size $n - 1$, and furthermore a full transversal if $n$ is odd. Montgomery \cite{M23} recently verified that this is true for large even $n$, a significant milestone for the existence question. An overview of the techniques used in \cite{M23} and related problems can be found in Montgomery's survey \cite{M24} on Latin transversals. 

Turning to enumeration, even for $\ZZ_n$ counting the number of transversals in the corresponding Latin square has historically been difficult. In 1991, Vardi \cite{V91} conjectured that for sufficiently large odd $n$, this number lies between $c_1^n n!$ and $c_2^n n!$ for appropriate constants $c_1,c_2\in(0,1)$. Let $T(n)$ denote the maximum possible number of transversals across \textit{all} order-$n$ Latin squares. McKay, McLeod, and Wanless \cite{MMW06} showed that $b^n \le T(n) \le c^n\sqrt n n!$ for $n \ge 5$, $b \approx 1.719$, and $c \approx 0.614$, thereby confirming Vardi's upper bound in the case of $\ZZ_n$. Later, Taranenko \cite{T15} improved the general upper bound by showing $T(n) \le \paren{(1+o(1))\frac{n}{e^2}}^n$. Finally, in 2016 Glebov and Luria \cite{GL16} used a probabilistic construction to prove a matching lower bound, thus establishing
\[
    T(n) = \paren{(1+o(1))\frac{n}{e^2}}^n.
\]

Along similar lines, let $t(n)$ denote the minimum number of transversals in an order-$n$ Latin square. Analogous questions can be asked about the asymptotics of $t(n)$. The example given above of the Cayley table of $\ZZ_n$ shows $t(n) = 0$ when $n$ is even. In his surveys on Latin transversals \cite{Wa07survey, Wa11survey}, Wanless asked about the function $t(n)$, noting that no significant upper bounds on $t(n)$ for $n$ odd have been found. In \cite{GL16}, Glebov and Luria raised the following question: is it true that for $n$ odd, $t(n) = (1-o(1))^n\cdot T(n)$? As motivation for why this could be expected, both random Latin squares and Latin squares arising from certain group structures exhibit behavior consistent with the proposed relation, as described below. 

Kwan \cite{K20} studied the number of transversals in random Latin squares, showing that asymptotically almost surely a uniformly random order-$n$ Latin square has at least $\paren{(1-o(1))\frac{n}{e^2}}^n$ transversals. This bound was later tightened by Eberhard, Manners, and Mrazovi\'c \cite{EMM23} to 
\[
    \paren{e^{-1/2} + o(1)}\frac{(n!)^2}{n^n}.
\]
Together with Taranenko's result, both of these results imply that almost all order-$n$ Latin squares have $(1-o(1))^n\cdot T(n)$ transversals.

Additionally, for Latin squares arising from finite abelian groups (such as $\ZZ_n$), the number of transversals is also asymptotically close to $T(n)$.
Eberhard, Manners, and Mrazovi\'c \cite{EMM19} determined using Fourier-analytic methods the precise asymptotics for the number of transversals in Latin squares arising from finite abelian groups. In particular, they confirmed Vardi's conjecture by proving that for any finite abelian group $G$ of odd order $n$, the number of transversals in the Latin square corresponding to the Cayley table of $G$ is 
\[
    \paren{e^{-1/2} + o(1)}\frac{(n!)^2}{n^{n-1}}.
\]

In this paper, we answer the Glebov-Luria question in the negative in general, by exhibiting Latin squares with exponentially fewer transversals than $T(n)$.   
\begin{thm}\label{thm:transversals}
    For sufficiently large $n$ with $n \equiv 0 \pmod 3$, there exist order-$n$ Latin  squares with at most
    \[
        \paren{\frac{n}{e^{2.117}}}^n
    \]
    transversals.
\end{thm}

Call an entry of a Latin square \textit{transversal-free} if it does not appear in any transversal of that Latin square. For $n \equiv 0 \pmod 3$, Egan and Wanless \cite{EW12} constructed order-$n$ Latin squares with a linear-sized subrectangle consisting entirely of transversal-free entries. In our proof of Theorem \ref{thm:transversals}, we show via an entropy argument that this construction has exponentially fewer than $T(n)$ transversals. This aligns with the intuition that Latin squares with a significant number of transversal-free entries should not have too many transversals.

A remaining question is to determine a lower bound for $t(n)$ when $n$ is odd. Ghafari and Wanless \cite{GW24} asked whether there exist odd-order Latin squares with only one transversal, though they consider this unlikely. We believe that there exists a constant $C > 0$ such that every Latin square of odd order $n$ has at least $(n/C)^n$ transversals for sufficiently large $n$.

We present Theorem \ref{thm:transversals} as a corollary to the following more general result upper-bounding the number of perfect matchings in bipartite hypergraphs. This connection arises from the standard representation of Latin squares as $3$-uniform tripartite hypergraphs, in which transversals correspond to perfect matchings. 

A hypergraph is \textit{$k$-uniform} if every edge has exactly $k$ vertices. In this paper, $k$ will always be a constant independent of other parameters. The \textit{degree} of a vertex in a hypergraph is the number of edges containing it, and the \textit{codegree} of two vertices is the number of edges containing both. The maximum degree and maximum codegree of a hypergraph $\mc H$ is denoted $\Delta(\mc H)$ and $\Delta_2(\mc H)$, respectively. A \textit{matching} in a hypergraph is a collection of pairwise disjoint edges. A hypergraph is \textit{bipartite with bipartition $(A,B)$} if $(A,B)$ is a partition of the vertex set such that every edge has exactly one vertex in $A$. A matching $M$ in such a hypergraph is \textit{$A$-perfect} if for every $a\in A$, there exists $e\in M$ such that $a\in e$.

\begin{thm}\label{thm:main}
    Let $\mc H$ be a $(k+1)$-uniform bipartite hypergraph with bipartition $(A,B)$, and let $\rho = \frac{|B|}{|A|}$.  If the average degree of the vertices in $A$ is at most $q$, every vertex in $B$ has degree at most $D$, and the maximum codegree of $\mc H$ satisfies $\Delta_2(\mc H) = o(q)$, then the number of $A$-perfect matchings in $\mc H$ is at most
    \[
        \exp\paren{|A|\int_0^1 \log\paren{(\rho-k)D+q(o(1)+x^k)}\mathrm dx}.
    \]
\end{thm}

Theorem \ref{thm:main} is a bipartite analogue of a theorem of Luria \cite[Theorem 3.1]{L17}, which uses entropy to upper-bound the number of perfect matchings in regular hypergraphs. Our setting allows for irregular degrees and imposes bipartite structure, giving rise to new applications.

Another corollary of Theorem \ref{thm:main} concerns proper edge-colorings of hypergraphs. A \textit{proper edge-coloring} of a hypergraph $\mc H$ is an assignment of colors to its edges such that no two edges of the same color share a vertex. The \textit{chromatic index} of $\mc H$, denoted $\chi'(\mc H)$, is the minimum number of colors used by a proper edge-coloring of $\mc H$.

Classically, Vizing's theorem \cite{V64} guarantees that every graph $G$ of maximum degree $\Delta(G)$ has chromatic index at most $\Delta(G) + 1$. In the hypergraph setting, an analogous result of Pippenger and Spencer \cite{PS89} shows that for $k$-uniform hypergraphs $\mc H$ with small codegrees, the chromatic index asymptotically satisfies 
\[
    \chi'(\mc H) \le (1 + o(1))\Delta(\mc H).
\]
Later, Kahn \cite{K96} showed that this asymptotic behavior also extends to list colorings.

The Pippenger-Spencer theorem naturally raises an enumerative question: how many proper edge-colorings exist when the number of available colors $q$ is close to the above chromatic index? By encoding edge-colorings as perfect matchings in an auxiliary bipartite hypergraph, Theorem $\ref{thm:main}$ gives an upper bound on the number of such colorings in the case of regular hypergraphs, which we state as Theorem \ref{thm:colorings}.

\begin{thm}\label{thm:colorings}
    Let $\mc G$ be a $k$-uniform $D$-regular hypergraph on $n$ vertices, and let $q = (1+o(1))D$. If the maximum codegree of $\mc G$ satisfies $\Delta_2(\mc G) = o(q)$, then the number of proper $q$-edge-colorings of $\mc G$ is at most
    \[
        \paren{(1+o(1))\frac{q}{e^k}}^{Dn/k}.
    \]
\end{thm} 

For the graph case $k = 2$, if $2D \mid n$ and $q = (1+o(1))D$, then the bound in Theorem \ref{thm:colorings} is asymptotically tight. By the theorem, every $D$-regular graph on $n$ vertices admits at most $((1+o(1))D/e^2)^{Dn/2}$ proper $q$-edge-colorings. A matching lower bound is realized by the graph $G$ consisting of a disjoint union of $n/(2D)$ copies of $K_{D,D}$. Proper $D$-edge-colorings of $K_{D,D}$ correspond to order-$D$ Latin squares. By a result of Egorychev \cite{E81} and, independently, Falikman \cite{F81}, each copy of $K_{D,D}$ has at least $((1-o(1))D/e^2)^{D^2}$ proper $D$-edge-colorings. Since $q = (1 + o(1))D$, there are therefore at least $((1-o(1))D/e^2)^{Dn/2}$ proper $q$-edge-colorings. This draws an interesting parallel to the Upper Matching Conjecture \cite{FK08}, which states that for any $n$, $k$, and $D$ satisfying $2D \mid n$, the $n$-vertex $D$-regular graph which maximizes the number of matchings of size $k$ is the disjoint union of $n/(2D)$ $K_{D,D}$'s. Davies, Jenssen, and Perkins \cite{DJP21} recently confirmed the conjecture asymptotically. In light of this, it would be interesting to determine which $n$-vertex $D$-regular graph has the maximum number of proper $q$-edge-colorings for $q \ge D$, and in particular whether this maximum is also achieved by the disjoint union of $K_{D,D}$'s. Applying Theorem \ref{thm:colorings} with $\mathcal G = K_{D,D}$, we also get that the number of order-$D$ Latin squares is at most $((1\pm o(1))D/e^2)^{D^2}$, which recovers the upper bound of van Lint and Wilson \cite{VLW01}.

In Section \ref{sec:entropy}, we present the basic entropy tools used in our analysis. Section \ref{sec:proofmain} contains the proof of Theorem~\ref{thm:main}, which bounds the number of perfect matchings in bipartite hypergraphs. In Section \ref{sec:app}, we apply this result to prove Theorem \ref{thm:transversals} and Theorem \ref{thm:colorings}, concerning Latin squares with few transversals and upper bounds for hypergraph colorings.

\section{The Entropy Method}\label{sec:entropy}
Here we present a basic overview of the entropy method, a probabilistic approach to counting via the Shannon entropy function $H(X)$. This is a concept from information theory, an introduction to which can be found in the textbook by Cover and Thomas \cite{CT06}. In the field of combinatorial design theory, entropy has been used to prove upper bounds on the count of various objects \cite{L17, SIM23, K18, LL13, K20, CR11}. For instance, Luria \cite{L17} used entropy to prove an asymptotic upper bound on the number of perfect matchings in a $d$-regular, $k$-uniform hypergraph. As a corollary, this extends to bounds on the number of $n$-queens configurations and Steiner systems, among others. In \cite{SIM23}, Simkin improved the upper bound on the number of $n$-queens configurations by combining entropy with limit objects. 

Let $X$ be a discrete random variable with support $\supp X$, and let $p(x) := \prob{X = x}$. Define the \textit{entropy} $H(X)$ of $X$ to be the quantity
\[
    H(X) = \sum_{x \in \supp X} -p(x)\log_2 p(x).
\]
Intuitively, $H(X)$ captures the amount of surprise information or uncertainty of $X$. In other words, $H(X)$ measures the degree of randomness of $X$. Hence when $X$ is uniform (very random), $H(X)$ should be large, and when $X$ is constant then $H(X) = 0$. This intuition is captured in the following lemma, which we reference as the \textit{uniform bound}.
\begin{lemma}[Uniform Bound]\label{lem:uniform}
    Let $X$ be a discrete random variable. Then $H(X) \leq \log_2|\mathrm{supp}X|$, with equality if and only if $X$ is uniformly distributed.
\end{lemma}
\noindent For jointly distributed random variables $X$ and $Y$, define the \textit{joint entropy} of $X$ and $Y$ as
\[
    H(X,Y) = \sum_{x \in \supp X, y \in \supp Y} -p(x,y)\log_2p(x,y),
\]
and define the \textit{conditional entropy} of $X$ given $Y$ as
\[
    H(X \mid Y) = \sum_{y \in \supp Y}p(y) \sum_{x \in \supp X} - p(x \mid y) \log_2 p(x \mid y).
\]
The quantity $H(X \mid Y)$ captures the amount of surprise information of $X$ given that $Y$ has already been revealed. 

A fundamental identity related to conditional entropy is the \textit{chain rule}, which describes how the entropy of a random vector can be analyzed by uncovering its components sequentially, one at a time. For example, to analyze a random Latin square using entropy, we can think of its entries as forming a random vector of $n^2$ elements, then use the chain rule to reveal the square entry-wise.
\begin{lemma}[Chain Rule]\label{lem:chain}
    For random variables $X_1, \dots, X_n$,
    \[
        H(X_1, \dots, X_n) = \sum_{i = 1}^n H(X_i \mid X_1, \dots, X_{i-1}).
    \]
\end{lemma}
For a broader introduction and additional applications of the entropy method in combinatorics, see Chapter 15 of Alon-Spencer \cite{AS16}.

\section{Proof of the main theorem}\label{sec:proofmain}

In this section, we prove Theorem~\ref{thm:main} using the entropy method.  In the proof, we choose an $A$-perfect matching $X$ of $\mc H$ uniformly at random and bound the entropy $H(X)$ using the chain rule (Lemma~\ref{lem:chain}), by revealing the edges of $X$ one-by-one, in a uniformly random order.  If $\mc H$ is a bipartite hypergraph with bipartition $(A, B)$ and $X$ is an $A$-perfect matching of $\mc H$, we let $X_a$ denote the edge of $X$ containing $a$.
As we reveal the edges, we should expect the conditional entropy to decrease, as for each $a \in A$, there will be fewer possible values for $X_a$. 
This part of the argument is captured by the following lemma.  

\begin{lemma}\label{lemma:entropy-in-main-thm}
    Let $\mc H$ be a bipartite hypergraph with bipartition $(A,B)$. If $X$ is a random $A$-perfect matching of $\mc H$, then
    \[
        H(X)\le \expectc{X}{|A|\int_0^1\log_2\paren{\frac1{|A|}\sum_{a\in A}\sum_{e\ni a}x^{|\{a'\neq a:X_{a'} \cap e\neq\emptyset\}|}}\mathrm dx}.
    \]
\end{lemma}

\begin{proof}
Let $\alpha = (\alpha_a)_{a\in A}\in[0,1]^{|A|}$ be a vector of real numbers where for each $a\in A$, $\alpha_a$ is chosen from $[0,1]$ uniformly and independently at random. Since $\mc H$ is bipartite, we can write $X$ as the joint distribution $X = (X_a)_{a \in A}$.
We proceed by revealing each $X_a$ in increasing order of $\alpha_a$. As $H(X)$ is independent of $\alpha$, by the chain rule (Lemma \ref{lem:chain}) we have
\[
    H(X)=\expectc{\alpha}{H(X)}=\expectc{\alpha}{\sum_{a \in A}H(X_a \mid X_{a'} : \alpha_{a'} < \alpha_a)}.
\]
For each $a \in A$ and $e \ni a$, we say the assignment $X_a = e$ is \textit{legal} if every $v \in e\cap B$ satisfies
\[
    v \notin \bigcup_{a' : \alpha_{a'} < \alpha_a}X_{a'}.
\]
For each $a \in A$, let $N_a(X, \alpha)$ be the number of legal choices for $X_a$. Since $X_a$ is contained in the set of legal edges, and this set depends only on $X_{a'}$ for $a' \in A$ with $\alpha_{a'}<\alpha_a$, by the uniform bound, for every $\alpha$,
\[
    H(X_a\mid X_{a'} : \alpha_{a'}<\alpha_a)\le \expectc{X}{\log_2 N_a(X,\alpha)}.
\]
Combining the previous inequalities, by the linearity of expectation and Fubini's theorem, we have
\begin{equation*}
    H(X) \leq \expectc{\alpha}{\sum_{a \in A}\expectc{X}{\log_2 N_a(X,\alpha)}} = 
    \expectc{X}{\sum_{a \in A}\expectc{\alpha}{\log_2 N_a(X,\alpha)}}.
\end{equation*}
For each vertex $a \in A$, we now integrate to condition on a fixed evaluation of $\alpha_a$.
By the law of total expectation, for every $X$,
\begin{equation*}
    \expectc{\alpha}{\log_2N_a(X,\alpha)} = \int_0^1\log_2\expectc{\alpha\mid\alpha_a=x}{N_a(X,\alpha)}\mathrm dx,
\end{equation*}
and for every $x \in [0,1]$, we have
\begin{equation*}
    \log_2\expectc{\alpha\mid\alpha_a=x}{N_a(X,\alpha)} = \log_2\paren{\sum_{e \ni a}\PP_{\alpha\mid\alpha_a=x}[X_a = e\text{ is legal}]}
\end{equation*}
by the linearity of expectation. Given $\alpha_a = x$, the probability that the assignment $X_a = e$ is legal is the probability that $\alpha_{a'} > \alpha_a$ for every $a'$ such that $X_{a'}\cap e\neq\emptyset$. Therefore, for every $a \in A$ and $x \in [0,1]$, we have 
\[  
    \PP_{\alpha\mid\alpha_a=x}[X_a = e\text{ is legal}] = (1-x)^{|\{a'\neq a:X_{a'} \cap e\neq\emptyset\}|}.
\]
Altogether, substituting $1-x\mapsto x$ in the integral, we have
\begin{equation*}
    H(X) \leq \expectc{X}{\sum_{a \in A}\int_0^1\log_2\left(\sum_{e \ni a}x^{|\{a'\neq a : X_{a'}\cap e \neq\emptyset\}|}\right)\mathrm dx},
\end{equation*}
and the result follows 
by Jensen's inequality.
\end{proof}

In the proof of Theorem~\ref{thm:main}, we will use that, for ``most'' edges $e$, the quantity $|\{a' \neq a : X_{a'}\cap e\neq\emptyset\}|$ in Lemma~\ref{lemma:entropy-in-main-thm} is equal to $k$. 
This equality holds for \textit{all} edges when $\Delta_2(\mathcal H) = 1$ and $\rho = k$ (and thus $X$ is a perfect matching), but it does not hold in general.
In particular, an edge $e \notin X$ can be incident with fewer than $k$ edges in $X$ in one of two ways: either an edge in $X$ contains multiple vertices of $e$ in $B$, or $e$ contains vertices not saturated by $X$. 
To that end, we define the following sets of ``bad edges'' (shown in Figure \ref{fig:badedges}).

\begin{figure}
    \centering
    \resizebox{0.25\textwidth}{!}{
    \begin{tikzpicture}
        \draw[line width=2pt, rounded corners] (0,1) rectangle (1,6);
        \draw[line width=2pt, rounded corners] (2,0) rectangle (4,6);

        \node at (0.5, 6.4) {\Large $A$};
        \node at (3, 6.4) {\Large $B$};
        
        \foreach \i in {2,...,6} {
            \node[fill=black, circle, inner sep=2pt] (L\i) at (0.5, \i - 0.5) {};
        }
        \foreach \i in {1,...,6} {
            \node[fill=black, circle, inner sep=2pt] (R\i) at (2.5, \i - 0.5) {};
            \node[fill=black, circle, inner sep=2pt] (RR\i) at (3.5, \i - 0.5) {};
        }

        \begin{scope}[on background layer] 
            
            \foreach \i in {2,...,6} {
                \filldraw[fill=gray, fill opacity=0.3, draw=black, draw opacity=0.8, line width=1.5pt, rounded corners=5pt]
                ($(L\i) + (-0.3, -0.3)$) --
                ($(L\i) + (-0.3, 0.3)$) --
                ($(RR\i) + (0.3, 0.3)$) --
                ($(RR\i) + (0.3, -0.3)$) --
                cycle;
            }

            \filldraw[fill=blue, fill opacity=0.4, draw=blue, draw opacity=0.8, line width=1.5pt, rounded corners=5pt]
                ($(L5) + (0.07082, 0.3)$) --
                ($(R4) + (0.07082, 0.3)$) --
                ($(RR4) + (0.3, 0.3)$) --
                ($(RR4) + (0.3, -0.3)$) --
                ($(R4) + (-0.07082, -0.3)$) --
                ($(L5) + (0, -0.3)$) --
                ($(L5) + (-0.3, -0.3)$) --
                ($(L5) + (-0.3, 0.3)$) --
                cycle;

            \filldraw[fill=red, fill opacity=0.4, draw=red, draw opacity=0.8, line width=1.5pt, rounded corners=5pt]
                ($(L2) + (0.07082, 0.3)$) --
                ($(R1) + (0.07082, 0.3)$) --
                ($(RR1) + (0.3, 0.3)$) --
                ($(RR1) + (0.3, -0.3)$) --
                ($(R1) + (-0.07082, -0.3)$) --
                ($(L2) + (0, -0.3)$) --
                ($(L2) + (-0.3, -0.3)$) --
                ($(L2) + (-0.3, 0.3)$) --
                cycle;
        \end{scope}
    \end{tikzpicture}
    }
    \caption{Edges in $X$, $S(X)$, and $T(X)$ are\\ shown in black, blue, and red, respectively.}
    \label{fig:badedges}
\end{figure}
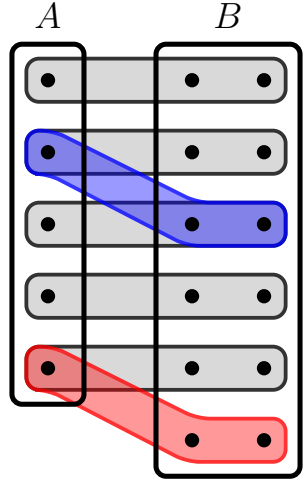

\begin{define}
    Let $\mc H$ be a bipartite hypergraph with bipartition $(A,B)$.  For an $A$-perfect matching $X$ of $\mc H$, we define
    \begin{itemize}[topsep=.3em, itemsep=0.3em]
        \item $S(X) := \{e \in E(\mc H) : \exists a \in A\text{ s.t. } e \neq X_a \text{ and } |X_a \cap e \cap B| \ge 2\}$, and
        \item $T(X) := \{e \in E(\mc H) : \exists v \in e \text{ s.t. } \forall a \in A, v \notin X_a\}$.
    \end{itemize}
    For each $a \in A$ let $S_a(X) := \{e \in S(X) : e \ni a\}$ and $T_a(X) := \{e \in T(X) : e \ni a\}$. 
 \end{define}

Nevertheless, we can show that under the hypotheses of Theorem~\ref{thm:main}, these sets of bad edges are not too large, as follows.

\begin{lemma}\label{lemma:bad-edge-bound}
    Let $\mc H$ be a $(k + 1)$-uniform bipartite hypergraph with bipartition $(A,B)$.  If $X$ is an $A$-perfect matching of $\mc H$, then
    \begin{equation*}
        \sum_{a\in A}|S_a(X)| = |S(X)| \leq |A|\binom{k}{2}(\Delta_2(\mc H) - 1),
    \end{equation*}
    and if every vertex of $B$ has degree at most $D$, then
    \begin{equation*}
        \sum_{a\in A}|T_a(X)| = |T(X)| \leq (|B| - k|A|)D.
    \end{equation*}
\end{lemma}
\begin{proof}
    To bound $|S_a(X)|$, fix some $a \in A$, and let $e$ be the edge in $X$ containing $a$. Since $e$ is $(k + 1)$ uniform, there are $k$ vertices in $e \cap B$, which determine $\binom k2$ unordered pairs. Each such pair is contained in at most $\Delta_2(\mc H)$ edges of $\mc H$, one of which is $e$ itself. By definition, each edge in $S_a(X)$ must contain at least one such pair, and hence $|S_a(X)| \le \binom k2(\Delta_2(\mc H) - 1)$.

    To bound $|T(X)|$, note that $X$ covers exactly $k|A|$ vertices in $B$, so $|B| - k|A|$ vertices of $B$ remain uncovered. Each uncovered vertex lies in at most $D$ edges of $\mc H$, and every edge in $T(X)$ contains at least one such uncovered vertex. Therefore, $|T(X)| \le (|B| - k|A|)D$.
\end{proof}

Now we can prove Theorem~\ref{thm:main}.

\begin{proof}[Proof of Theorem~\ref{thm:main}]
Let $\mc M$ be the set of all $A$-perfect matchings of $\mc H$, and choose some $X \in \mc M$ uniformly at random. By the uniform bound (Lemma \ref{lem:uniform}), $\log_2|\mc M| = H(X)$, so it suffices to upper-bound $H(X)$. By Lemma~\ref{lemma:entropy-in-main-thm},
\[
    H(X)\le \expectc{X}{|A|\int_0^1\log_2\paren{\frac1{|A|}\sum_{a\in A}\sum_{e\ni a}x^{|\{a'\neq a:X_{a'} \cap e\neq\emptyset\}|}}\mathrm dx}.
\]
Now fix some $a \in A$ of degree $\deg(a)$ and consider some arbitrary edge $e$ containing $a$. If $e = X_a$ or $e \in S_a(X) \cup T_a(X)$, we upper-bound the term $x^{|\{a'\neq a:X_{a'} \cap e\neq\emptyset\}|}$ in the summation by $1$. Otherwise, if $e\notin \{X_a\}\cup S_a(X)\cup T_a(X)$, then $|\{a'\neq a : X_{a'}\cap e\neq\emptyset\}| = k$. It follows that for every $a \in A$ and $x \in [0,1]$, 
\[
    \sum_{e \ni a}x^{|\{a'\neq a:X_{a'} \cap e\neq\emptyset\}|} \le 1 + |S_a(X)| + |T_a(X)| + \deg(a)x^k.
\]
Since $\Delta_2(\mc H) = o(q)$ and every vertex of $B$ has degree at most $D$, by Lemma~\ref{lemma:bad-edge-bound}, we have $|S(X)| \le |A|\binom k2(\Delta_2(\mc H)-1) = |A|o(q)$ and $|T(X)| \le |A|(\rho - k)D$. Therefore,
\begin{align*}
    H(X) &\le \expectc{X}{|A|\int_0^1 \log_2\paren{1 + \frac1{|A|}\sum_{a \in A}|S_a(X)| + \frac1{|A|}\sum_{a \in A}|T_a(X)| + \frac1{|A|}\sum_{a \in A}\deg(a)x^k}\mathrm dx}\\
    &= \expectc{X}{|A|\int_0^1 \log_2\paren{1 + o(q) + (\rho - k)D + qx^k}\mathrm dx}
    \\ &= \frac{|A|}{\log2}\int_0^1\log\paren{(\rho - k)D + q(o(1) + x^k)}\mathrm dx.
\end{align*}
Raising this to the power of $2$ gives that the number of $A$-perfect matchings of $\mathcal H$ is at most
\[
    |\mc M| \le \exp\paren{|A|\int_0^1\log\paren{(\rho - k)D + q(o(1) + x^k)}\mathrm dx},
\]
as desired.
\end{proof}

\section{Applications}\label{sec:app}

In this section, we use Theorem~\ref{thm:main} to prove Theorems \ref{thm:transversals} and \ref{thm:colorings}. Both applications involve evaluating integrals of the form appearing in the exponent of Theorem \ref{thm:main}, the calculations for which are isolated in the following lemmas.

\begin{lemma}\label{lemma:logintegral}
    For $k \ge 2$ constant and $\eps\to0$, 
    \begin{equation*}
        \int_0^1\log(\eps+x^k)\mathrm dx = -k + O(\eps^{1/k}).
    \end{equation*}
\end{lemma}
\begin{proof}
    First, note that
    \begin{align*}
        \int_0^1\log(\eps+x^k)\mathrm dx &= \int_0^1\log(x^k)\mathrm dx + \int_0^1 \log\paren{1 + \frac{\eps}{x^k}}\mathrm dx\\
        &= -k + \int_0^1\log\paren{1 + \frac{\eps}{x^k}}\mathrm dx\\
        &= -k + \int_0^{\eps^{1/k}}\log\paren{1 + \frac{\eps}{x^k}}\mathrm dx + \int_{\eps^{1/k}}^1\log\paren{1 + \frac{\eps}{x^k}}\mathrm dx. 
    \end{align*}

    Substituting $x = \eps^{1/k}t$ in the first integral above, we have
    \[
        \int_0^{\eps^{1/k}} \log\paren{1 + \frac{\eps}{x^k}}\mathrm dx = \eps^{1/k}\int_0^1\log\paren{1 + t^{-k}}\mathrm dt = O(\eps^{1/k}),
    \]
    since $\int_0^1\log(1+t^{-k})\mathrm dt$ is a positive constant depending only on $k$. Using the inequality $\ln(1+a)\le a$ for $a \ge 0$, we have
    \[
        \int_{\eps^{1/k}}^1 \log\paren{1 + \frac{\eps}{x^k}}\mathrm dx \le \int_{\eps^{1/k}}^1 \frac{\eps}{x^k}\mathrm dx = \frac{1}{k-1}\paren{\eps^{1/k} - \eps} = O(\eps^{1/k}).
    \]

    The desired result follows by combining the three equations above.
\end{proof}

\begin{lemma}\label{lemma:lsintegral}
    For $n$ sufficiently large, \begin{equation*}
    \int_0^1\log\paren{\paren{\frac89n+\frac13}(o(1)+x^2)}\mathrm dx \le \log n - 2.117.
    \end{equation*}
\end{lemma}
\begin{proof}
    By Lemma \ref{lemma:logintegral},
    \begin{align*}
    \int_0^1\log\paren{\paren{\frac89n+\frac13}(o(1)+x^2)}\mathrm dx &= \log\paren{\frac89n+\frac13} + \int_0^1\log\paren{o(1)+x^2}\mathrm dx\\
    &= \log n + \log(8/9) + o(1) - 2 + o(1) \\ &= \log n + \log(8/9) - 2 + o(1) \le \log n - 2.117.
       \qedhere
    \end{align*}
\end{proof}

\begin{lemma}\label{lemma:ecintegral}
    For $k \ge 2$ constant and $q = (1 + o(1))D$ as $q\to\infty$,
    \begin{equation*}
        \int_0^1 \log\paren{o(1)kD+q(o(1)+x^k)}\mathrm dx = \log q - k + o(1).
    \end{equation*}
\end{lemma}
\begin{proof}
    Since $k$ is constant, $k\cdot o(1) = o(1)$. By Lemma \ref{lemma:logintegral},
    \begin{align*}
        \int_0^1 \log\paren{o(1)kD+q(o(1)+x^k)}\mathrm dx  &= \int_0^1\log\paren{o(1)D + q(o(1) + x^k)}\mathrm dx\\
        &= \int_0^1\log\paren{q(o(1)+x^k)}\mathrm dx \\ &= \log q + \int_0^1 \log\paren{o(1)+x^k}\mathrm dx = \log q - k + o(1).
        \qedhere
    \end{align*}
\end{proof}

\subsection{Latin squares with few transversals}\label{subsec:transversal}
\begin{proof}[Proof of Theorem \ref{thm:transversals}]
Recall that an entry of a Latin square is \textit{transversal-free} if it does not appear in any transversal of that Latin square. Choose some sufficiently large integer $n$ with $n \equiv 0 \pmod 3$. When $n$ is even, the Cayley table of $\ZZ/n\ZZ$ is an order-$n$ Latin square containing no transversals, so it suffices to consider the case where $n$ is odd. By a theorem of Egan and Wanless \cite{EW12}, there exists an order-$n$ Latin square $L$ containing an $\paren{\frac n3-1} \times \frac n3$ Latin subrectangle consisting entirely of transversal-free entries. Let $R$, $C$, and $S$ be disjoint sets indexing the rows, columns, and symbols of $L$, respectively. We now construct an auxiliary hypergraph $\mc H$ as follows. Let $V(\mc H) = R \cup C \cup S$ and for each $(r,c,s) \in R \times C \times S$, let $\{r,c,s\} \in E(\mc H)$ if $L_{r,c} = s$ and the entry $(r,c,s) \in L$ is not transversal-free. Note that $\mc H$ is $3$-uniform and bipartite with bipartition $(R, C \cup S)$. Since $L$ is a Latin square, the maximum degree of $\mc H$ is $n$ and the maximum codegree of $\mc H$ is $1$. By our choice of $L$ and the definition of $E(\mc H)$, the average degree of vertices in $R$ is at most
\[
    \frac1n\left[\paren{\frac{2n}{3} + 1}n + \paren{\frac n3 - 1}\frac{2n}{3}\right] = \frac89n + \frac13.
\]
By construction, $E(\mc H)$ omits only the transversal-free entries of $L$, and therefore the property that all transversals of $L$ correspond to $R$-perfect matchings in $\mc H$ is preserved. By Theorem \ref{thm:main} with $A = R, \rho = k = 2$, $D = n$, and $q = \frac89n + \frac13$, we have that the number of transversals of $L$ is at most
\begin{align*}
    \exp\paren{n\int_0^1\log\paren{\paren{\frac89n+\frac13}(o(1) + x^2)}\mathrm dx} \le \paren{\frac{n}{e^{2.117}}}^n,
\end{align*}
where the final inequality holds by Lemma \ref{lemma:lsintegral}.
\end{proof}

The authors remark that the value of $2.117$ could be improved slightly by performing Jensen's inequality at a later stage in the proof of Theorem \ref{thm:main}.

\subsection{Edge colorings of hypergraphs}\label{subsec:coloring}

\begin{proof}[Proof of Theorem \ref{thm:colorings}]
From $\mc G$ we form an \textit{incidence hypergraph} $\mc H$ as follows. Let $V(\mc H) = E(\mc G) \cup (V(\mc G) \times [q])$, and for each $c \in [q]$ and $e \in E(\mc G)$, add the edge $\{e\} \cup (e \times \{c\})$ to $E(\mc H)$. Note that $\mc H$ is bipartite with bipartition $(E(\mc G), V(\mc G) \times [q])$ and $(k + 1)$-uniform. Since $\mc G$ is $D$-regular and $k$-uniform, we have $e(\mc G) = \frac{nD}{k}$, and hence 
\[
    \frac{|V(\mc G) \times [q]|}{|E(\mc G)|} = \frac{qn}{nD/k} = (1+o(1))k.
\]

For every $e \in E(\mc G)$, we have $\deg_{\mc H}(e) = q$, and the maximum degree of any vertex in $V(\mc G)\times [q]$ is $\Delta(\mc G) = D$. Moreover, the maximum codegree of $\mc H$ satisfies $\Delta_2(\mc H) \le \Delta_2(\mc G) = o(q)$. Finally, a proper $q$-edge-coloring of $\mc G$ corresponds to an $(E(\mc H))$-perfect matching in $\mc H$. Applying Theorem \ref{thm:main} with $A = E(\mc G)$ and $\rho = (1+o(1))k$, the number of proper $q$-edge-colorings of $\mc G$ is at most
\[
    \exp\paren{\frac{nD}{k}\int_0^1\log\paren{((1+o(1))k - k)D + q(o(1) + x^k)}\mathrm dx} = \paren{(1+o(1))\frac{q}{e^k}}^{Dn/k},
\]
where the final inequality holds by Lemma \ref{lemma:ecintegral}.
\end{proof}

\providecommand{\bysame}{\leavevmode\hbox to3em{\hrulefill}\thinspace}
\providecommand{\MR}{\relax\ifhmode\unskip\space\fi MR }
\providecommand{\MRhref}[2]{%
  \href{http://www.ams.org/mathscinet-getitem?mr=#1}{#2}
}


\begin{thebibliography}{10}

\bibitem{AS16}
N.~Alon and J.~H. Spencer, \emph{The probabilistic method}, 4th ed., Wiley Publishing, 2016.

\bibitem{BR91}
R.~A. Brualdi and H.~J. Ryser, \emph{Combinatorial matrix theory}, Cambridge University Press, 1991.

\bibitem{CT06}
T.~M. Cover and J.~A. Thomas, \emph{Elements of information theory}, second ed., John Wiley \& Sons, 2006. \MR{2239987}

\bibitem{CR11}
J.~Cutler and A.~J. Radcliffe, \emph{An entropy proof of the {K}ahn-{L}ov\'asz theorem}, Electron. J. Combin. \textbf{18} (2011), Paper 10, 9. \MR{2770115}

\bibitem{DJP21}
E.~Davies, M.~Jenssen, and W.~Perkins, \emph{A proof of the upper matching conjecture for large graphs}, J. Combin. Theory Ser. B \textbf{151} (2021), 393--416. \MR{4299069}

\bibitem{EMM19}
S.~Eberhard, F.~Manners, and R.~Mrazovi\'c, \emph{Additive triples of bijections, or the toroidal semiqueens problem}, J. Eur. Math. Soc. (JEMS) \textbf{21} (2019), 441--463. \MR{3896207}

\bibitem{EMM23}
\bysame, \emph{Transversals in quasirandom latin squares}, Proc. Lond. Math. Soc. (3) \textbf{127} (2023), 84--115. \MR{4611404}

\bibitem{EW12}
J.~Egan and I.~M. Wanless, \emph{Latin squares with restricted transversals}, J. Combin. Des. \textbf{20} (2012), 124--141. \MR{2868854}

\bibitem{E81}
G.~P. Egorychev, \emph{The solution of van der {W}aerden's problem for permanents}, Adv. in Math. \textbf{42} (1981), 299--305. \MR{642395}

\bibitem{F81}
D.~I. Falikman, \emph{Proof of the van der {W}aerden conjecture on the permanent of a doubly stochastic matrix}, Mat. Zametki \textbf{29} (1981), 931--938, 957. \MR{625097}

\bibitem{FK08}
S.~Friedland, E.~Krop, and K.~Markstr\"om, \emph{On the number of matchings in regular graphs}, Electron. J. Combin. \textbf{15} (2008), Research Paper 110, 28. \MR{2438582}

\bibitem{GW24}
A.~Ghafari and I.~M. Wanless, \emph{Latin squares whose transversals share many entries}, arxiv:2412.12466 (2024).

\bibitem{GL16}
R.~Glebov and Z.~Luria, \emph{On the maximum number of {L}atin transversals}, J. Combin. Theory Ser. A \textbf{141} (2016), 136--146. \MR{3479241}

\bibitem{K96}
J.~Kahn, \emph{Asymptotically good list-colorings}, J. Combin. Theory Ser. A \textbf{73} (1996), 1--59. \MR{1367606}

\bibitem{K18}
P.~Keevash, \emph{Counting designs}, J. Eur. Math. Soc. (JEMS) \textbf{20} (2018), 903--927. \MR{3779688}

\bibitem{K20}
M.~Kwan, \emph{Almost all {S}teiner triple systems have perfect matchings}, Proc. Lond. Math. Soc. (3) \textbf{121} (2020), 1468--1495. \MR{4144368}

\bibitem{LL13}
N.~Linial and Z.~Luria, \emph{An upper bound on the number of {S}teiner triple systems}, Random Structures Algorithms \textbf{43} (2013), 399--406. \MR{3124689}

\bibitem{L17}
Z.~Luria, \emph{New bounds on the number of n-queens configurations}, arxiv:1705.05225 (2017).

\bibitem{MMW06}
B.~D. McKay, J.~C. McLeod, and I.~M. Wanless, \emph{The number of transversals in a {L}atin square}, Des. Codes Cryptogr. \textbf{40} (2006), 269--284. \MR{2251320}

\bibitem{M23}
R.~Montgomery, \emph{A proof of the {R}yser-{B}rualdi-{S}tein conjecture for large even $n$}, arxiv:2310.19779 (2023).

\bibitem{M24}
R.~Montgomery, \emph{Transversals in {L}atin squares}, Surveys in combinatorics 2024, London Math. Soc. Lecture Note Ser., vol. 493, Cambridge Univ. Press, Cambridge, 2024, 131--158. \MR{4749040}

\bibitem{PS89}
N.~Pippenger and J.~Spencer, \emph{Asymptotic behavior of the chromatic index for hypergraphs}, J. Combin. Theory Ser. A \textbf{51} (1989), 24--42. \MR{993646}

\bibitem{R67}
H.~Ryser, \emph{Neuere probleme der kombinatorik}, Vorträge über Kombinatorik, Oberwolfach (1967), 69--91.

\bibitem{SIM23}
M.~Simkin, \emph{The number of n-queens configurations}, Advances in Mathematics \textbf{427} (2023), 109127.

\bibitem{S75}
S.~K. Stein, \emph{Transversals of {L}atin squares and their generalizations}, Pacific J. Math. \textbf{59} (1975), 567--575. \MR{387083}

\bibitem{T15}
A.~A. Taranenko, \emph{Multidimensional permanents and an upper bound on the number of transversals in {L}atin squares}, J. Combin. Des. \textbf{23} (2015), 305--320. \MR{3350096}

\bibitem{VLW01}
J.~H. van Lint and R.~M. Wilson, \emph{A course in combinatorics}, 2 ed., Cambridge University Press, 2001.

\bibitem{V91}
I.~Vardi, \emph{Computational recreations in {M}athematica}, Addison-Wesley Publishing Company, Advanced Book Program, Redwood City, CA, 1991. \MR{1150054}

\bibitem{V64}
V.~G. Vizing, \emph{On an estimate of the chromatic class of a {$p$}-graph}, Diskret. Analiz (1964), 25--30. \MR{180505}

\bibitem{Wa07survey}
I.~M. Wanless, \emph{Transversals in {L}atin squares}, Quasigroups Related Systems \textbf{15} (2007), 169--190. \MR{2379130}

\bibitem{Wa11survey}
\bysame, \emph{Transversals in {L}atin squares: a survey}, Surveys in combinatorics 2011, London Math. Soc. Lecture Note Ser., vol. 392, Cambridge Univ. Press, Cambridge, 2011, 403--437. \MR{2866738}

\end{thebibliography}
\end{document}